\documentclass[11pt]{article}

\usepackage{fullpage}
\usepackage{amsopn,amssymb,amsthm,amsmath}

\newcommand{\R}{\mathbb{R}}

\newcommand{\eps}{\varepsilon}

\newtheorem*{lemma*}{Lemma}

\title{A Short Proof for Gap Independence of Simultaneous Iteration}
\author{Edo Liberty\\ Yahoo Research}
\date\nonumber

\begin{document}
\maketitle
\pagenumbering{gobble}
\vspace{-.5cm}
This note provides a very short proof of a spectral gap independent property of the simultaneous iterations algorithm for finding the top singular space of a matrix \cite{RokhlinST09,HalkoMT2011,MuscoM15,WittenE15}. The proof is terse but completely self contained and should be accessible to the linear algebra savvy reader. 

\begin{lemma*} Let $A \in \R^{n \times m}$ be an arbitrary matrix and let $G \in \R^{m \times k}$ be a matrix of i.i.d.\ random Gaussian entries. 
Let $t = c\cdot \log(n/\eps)/\eps$ and $Z = \operatorname{span}((AA^T)^t A G)$ then with high probability depending only on a universal constant $c$
\[
||A - ZZ^TA|| \le (1+\eps)\sigma_{k+1}
\]
\end{lemma*}
\begin{proof}
$||A - ZZ^TA|| = \max_{x :\|x\|=1}  \|x^T A\|$ such that $\|x^TZ\| = 0$.
Using the SVD of $A$ we change variables $A = USV^T$, $x = Uy$ and $G' = V^TG$.
Note that $G'$ is also a matrix of i.id.\ Gaussian entries because $V$ is orthogonal.
We get $||A - ZZ^TA|| = \max_{y:\|y\|=1}  \|y^TS\|$ s.t.\ $y^TS^{2t+1}G' = 0$.
We now break $y$, $S$, and $G'$ to two blocks each such that
\[
y =
\left(\begin{array}{c}
y_1 \\ \hline
y_2 \\
\end{array}\right)
\mbox{,\;\;}
S = \left(\begin{array}{c|c}
S_1 & 0 \\ \hline
0 & S_2 \\
\end{array}\right)
\mbox{,\;\;}
G' = \left(\begin{array}{c}
G'_1  \\ \hline
G'_2 \\
\end{array}\right)
\] 
and $y_1 \in \R^{k}$, $y_2 \in \R^{n-k}$, $S_1 \in \R^{k \times k}$, $S_2 \in \R^{(n-k) \times (n-k)}$, $G'_1 \in \R^{k \times k}$, and $G'_2 \in \R^{(n-k) \times k}$.
\begin{eqnarray*}
0 &=& \|y^T S^{2t+1} G'\| = \|y_1^T S^{2t+1}_1 G'_1+  y_2^T S^{2t+1}_2 G'_2\| \\
&\ge& \|y_1^T S^{2t+1}_1 G'_1\| -  \|y_2^T S^{2t+1}_2 G'_2\| \\
&\ge& \|y_1^T S^{2t+1}_1\|/\|G'^{-1}_1\| - \|y_2^T\| \cdot \|S^{2t+1}_2 \| \cdot  \|G'_2\| \\
&\ge& |y_1(i)| \sigma_{i}^{2t+1}/\|G'^{-1}_1\| - \sigma_{k+1}^{2t+1} \cdot  \|G'_2\| \ .
\end{eqnarray*}
\noindent This gives that $|y_1(i)| \le (\sigma_{k+1}/\sigma_i)^{2t+1}\|G'_2\| \|G'^{-1}_1\|$. Equipped with this inequality we bound the expression $\|y^TS\|$.
Let $1 \le k' \le k$ be such that $\sigma_{k'} \ge (1+\eps)\sigma_{k+1}$ and $\sigma_{k'+1} < (1+\eps)\sigma_{k+1}$. If no such $k'$ exists the claim is trivial.
\begin{eqnarray}
||A - ZZ^TA||^2 &=& \|y^TS\|^2 = \sum_{i=1}^{k'}y^2_i \sigma_i^2 + \sum_{i=k'+1}^{n}y^2_i \sigma_i^2 \\
&\le& ( \|G'_2\|^2 \|G'^{-1}_1\|^2 \sum_{i=1}^{k'}(\sigma_{k+1}/\sigma_i)^{4t}  \sigma_{k+1}^2 ) + (1+\eps)\sigma_{k+1}^2 \\
&\le& \left[ \|G'_2\|^2 \|G'^{-1}_1\|^2 k (1/(1+\eps))^{4t} + (1+\eps)\right]\sigma_{k+1}^2 \le (1+2\eps)\sigma_{k+1}^2
\end{eqnarray}
The last step is correct as long as $ \|G'_2\|^2 \|G'^{-1}_1\|^2 k (1/(1+\eps))^{4t} \le \eps \sigma^2_{k+1}$ which holds for $t \ge \log(\|G'_2\|^2 \|G'^{-1}_1\|^2 k/\eps) /4\log(1+\eps) = O(\log(n/\eps)/\eps)$. The last inequality uses the fact that $G'_1$ and $G'_2$ are random gaussian due to rotational invariance of the Gaussian distribution. This means that $\|G'_2\|^2 \|G'^{-1}_1\|^2 = O(\operatorname{poly}(n))$ with high probability \cite{Rudelson08}.
Finally, $||A - ZZ^TA|| \le \sqrt{1+2\eps}\cdot\sigma_{k+1} \le (1+\eps)\sigma_{k+1}$.
\end{proof}


\end{document}